\documentclass[10pt]{article}
\font\smallit=cmti10

\usepackage{amssymb,latexsym,amsmath,epsfig,amsthm,graphicx,tikz-network,array,mathbbol,subcaption} 

\makeatletter

\renewcommand\section{\@startsection{section}{1}{\z@}
{-30pt \@plus -1ex \@minus -.2ex}
{2.3ex \@plus.2ex}
{\normalfont\normalsize\bfseries}}

\renewcommand\subsection{\@startsection{subsection}{2}{\z@}
{-3.25ex\@plus -1ex \@minus -.2ex}
{1.5ex \@plus .2ex}
{\normalfont\normalsize\bfseries}}

\renewcommand\subsubsection{\@startsection{subsubsection}{2}{\z@}
{-3.25ex\@plus -1ex \@minus -.2ex}
{1.5ex \@plus .2ex}
{\normalfont\normalsize\bfseries}}

\renewcommand{\@seccntformat}[1]{\csname the#1\endcsname. }

\makeatother

\begin{document}

\begin{center}
\uppercase{\bf A note on Hadwiger's conjecture: \\ Another proof that every $4$-chromatic graph has a $K_4$ minor}
\vskip 20pt
{\bf Daniel C. McDonald}\\
{\smallit Wolfram Research Inc., Champaign, Illinois}\\
{\tt daniel.cooper.mcdonald@gmail.com}\\
\end{center}
\vskip 30pt

\centerline{\bf Abstract}
\noindent
The first non-obvious case of Hadwiger's Conjecture states that every graph $G$ with chromatic number at least $4$ has a $K_4$ minor. We give a new proof that derives the $K_4$ minor from a proper $3$-coloring of a subgraph of $G$.
\pagestyle{myheadings} 
\thispagestyle{empty} 
\baselineskip=12.875pt 
\vskip 30pt

\newtheorem{thm}{Theorem}
\newtheorem{lem}[thm]{Lemma}
\newtheorem{prop}[thm]{Proposition}
\newtheorem{cor}[thm]{Corollary}
\newtheorem{hyp}[thm]{Hypothesis}
\newtheorem{case}[thm]{Case}
\theoremstyle{definition}
\newtheorem*{ack}{Acknowledgment}

\section{Introduction} 
\label{intro}
Hadwiger's conjecture states that every graph $G$ satisfying $\chi(G)\geq k$ has a $K_k$ minor (i.e. if $G$ is not $(k-1)$-colorable, then there exist $k$ pairwise disjoint nonempty connected subgraphs such that each pair is joined by an edge); see \cite{S} for an overview. The first three cases respectively state that $G$ has a vertex if $\chi(G)\geq 1$, $G$ has an edge if $\chi(G)\geq 2$, and $G$ has a cycle if $\chi(G)\geq 3$. The first non-obvious case is the one proved here: $G$ has a $K_4$ minor if $\chi(G)\geq 4$. This case was originally proved by Hadwiger in the 1943 paper \cite{H} in which he made his conjecture; he showed that every nonempty graph $G$ with no $K_4$ minor has a vertex $v$ of degree at most $2$, which implies that all such graphs are $3$-colorable by way of recursively giving $G-v$ a proper $3$-coloring, then coloring $v$ with a color not used on its neighbors. Dirac \cite{Di} proved the same statement in 1952 by noting that every graph $G$ satisfying $\chi(G)\geq 4$ contains a subgraph $H$ satisfying $\chi(H)=4$ with the property that $\chi(H-v)=3$ for every vertex $v$ in $H$, with this implying $H$ has minimum degree at least $3$ and no cut vertex; Dirac proved that $H$ having these properties implies $H$ has a $K_4$ minor. Duffin \cite{Du} once again proved this statement in 1965 by showing that the subgraphs of series-parallel graphs are precisely the graphs having no $K_4$ minor, which implies $3$-colorability since these graphs always have a vertex of degree at most $2$. Appel and Haken \cite{AH,AHK} proved in 1976 via their famed Four Color Theorem that $G$ has a $K_5$ minor if $\chi(G)\geq 5$, while Robertson, Seymour, and Thomas \cite{RST} used that result to prove in 1993 that $G$ has a $K_6$ minor if $\chi(G)\geq 6$; further cases remain open.

Our proof that $G$ has a $K_4$ minor if $\chi(G)\geq 4$, given in the next section, starts with a $4$-chromatic subgraph $H$ of $G$ having an edge $x y$ such that $\chi(H-x y)=3$. Giving $H-x y$ a proper $3$-coloring allows us to partition the vertices of $H$ into three sets, two of which are independent in $H$ and one of which would be independent except for the edge $x y$. We use this structure to construct a $K_4$ minor in $H$. We are not aware of this proof elsewhere in the literature. Extending this proof strategy to subsequent cases of Hadwiger's conjecture is left as an exercise to the reader.

\section{Theorem Statement and Proof}
\label{proof}

\begin{thm}
If $G$ satisfies $\chi(G)\geq 4$, then $G$ has a $K_4$ minor.
\end{thm}
\begin{proof}
Delete edges from $G$ until the resulting graph $H$ satisfies $\chi(H)=4$ and contains an edge $x y$ such that $\chi(H-x y)=3$. Color the vertices of $H$ (improperly) by giving $H-x y$ a proper $3$-coloring $\pi$, noting that $\pi(x)=\pi(y)$ (or else $\chi(H)\leq3$). Let $\pi$ partition the vertices of $H$ into the sets $\pi_1$, $\pi_2$, and $\pi_3$, with $\{x,y\}\subseteq \pi_1$.

Note that the subgraph $H-\pi_1$ of $H$ must contain some edge, or else $\chi(H)\leq3$ because $\pi$ could be changed into a proper $3$-coloring of $H$ by recoloring $\pi_3$ like $\pi_2$ and giving $y$ its own color. Further note that the subgraphs $H-\pi_3$ and $H-\pi_2$ of $H$ must each contain an odd cycle, or else one of these subgraphs would be bipartite, and hence $\chi(H)\leq3$ because $\pi$ could be changed into a proper $3$-coloring of $H$ by properly recoloring $H-\pi_3$ or $H-\pi_2$ with $2$ colors.

Let $\rho$ be the (improper) $3$-coloring of $H$ partitioning the vertices of $H$ into the sets $\rho_1$, $\rho_2$, and $\rho_3$, where $\rho$ is obtained from $\pi$ by swapping the colors in each component of $H-\pi_1$ not containing a vertex from an odd cycle in $H-\pi_3$; that is, let $\rho_1=\pi_1$, let $\rho_2$ be the union of the subset of vertices of $\pi_2$ in a component of $H-\pi_1$ containing a vertex from an odd cycle in $H-\pi_3$ and the subset of vertices of $\pi_3$ in a component of $H-\pi_1$ not containing a vertex from an odd cycle in $H-\pi_3$, and let $\rho_3$ be the union of the subset of vertices of $\pi_3$ in a component of $H-\pi_1$ containing a vertex from an odd cycle in $H-\pi_3$ and the subset of vertices of $\pi_2$ in a component of $H-\pi_1$ not containing a vertex from an odd cycle in $H-\pi_3$. Then the only monochromatic edge of $H$ under $\rho$ is again $x y$, and furthermore there is a component of $H-\rho_1$ that contains both a vertex from an odd cycle in $H-\rho_3$ and a vertex from an odd cycle in $H-\rho_2$. Indeed, the components of $H-\pi_1$ and $H-\rho_1$ are the same, every odd cycle in $H-\pi_3$ still exists in $H-\rho_3$, and every odd cycle in $H-\rho_2$, of which there must be at least one (or else $H-\rho_2$ would be bipartite, and hence $\chi(H)\leq3$ because $\rho$ could be changed into a proper $3$-coloring of $H$ by recoloring $H-\rho_2$ with $2$ colors), would contain a vertex from a component of $H-\rho_1$ containing a vertex from an odd cycle in $H-\rho_3$ (since any odd cycle contained in $H-(\rho_2\cup(\rho_3-\pi_3))$ would have been an odd cycle in $H-\pi_3$, contradicting our construction of $\rho$ by only swapping colors on vertices in components of $H-\pi_1$ that didn't contain any vertex from an odd cycle in $H-\pi_3$).

Let $C2$ and $C3$ be odd cycles in $H-\rho_3$ and $H-\rho_2$, respectively, such that there exist vertices $u$ in $C2$ and $v$ in $C3$ such that $u$ and $v$ are in the same component of $H-\rho_1$; choose $u$ and $v$ to be the closest such vertices in $H-\rho_1$, so that there exists a path $P$ between $u$ and $v$ in $H-\rho_1$ no internal vertex of which lies in either $C2$ or $C3$ (since $u\in C2-\rho_1$ and $v\in C3-\rho_1$, we also have $u\in\rho_2$ and $v\in\rho_3$ and thus $u\notin C3$ and $v\notin C2$). Note that $x y$ is an edge in both $C2$ and $C3$, as $H-(\rho_3\cup\{x y\})$ and $H-(\rho_2\cup\{x y\})$ are both bipartite and therefore lack odd cycles. Let $Q$ be the shortest path subgraph of $C3$ containing $v$ whose endpoints $w$ and $z$ lie in $C2$ (so no internal vertices of $Q$ are in $C2$); $Q$ must exist because $C3-\{x y\}$ is a path subgraph of $C3$ containing $v$ whose endpoints $x$ and $y$ lie in $C2$.

We complete the proof by constructing four pairwise disjoint nonempty connected subgraphs $H1,H2,H3,H4$ of $H$ such that each pair is joined by an edge; these provide a $K_4$ minor in $G$. If $w$ and $z$ lie in separate components of $C2-\{x y, u\}$, define the edge $e$ by $e=x y$, and if $w$ and $z$ lie in the same component of $C2-\{x y, u\}$, without loss of generality assume the vertices $x,w,z,u,y$ appear in that order (not necessarily all consecutively or distinctly) in $C3$, and pick the edge $e$ to be somewhere between $w$ and $z$ in this ordering. Let our subgraphs be $H1=Q-\{w,z\}$, $H2=P-\{v\}$, and the two components $H3$ and $H4$ of $C2-\{e, u\}$; see Figure \ref{subgraphs}. Note that $H1$ is nonempty and connected because $w$ and $z$ are the endpoints of the path $Q$ containing the internal vertex $v$, $H2$ is nonempty and connected because $v$ is an endpoint of the path $P$ between distinct vertices $u$ and $v$, and $H3$ and $H4$ are each nonempty and connected because $C2$ is a cycle containing the edge $e$ and the vertex $u$ not in $e$. Each pair of $H2$, $H3$, and $H4$ has an edge joining them but no vertex in common because the vertices of the cycle $C2$ are partitioned among the vertex sets of the paths $H3$ and $H4$ along with the only vertex $u$ in $H2$ that also lies in $C2$. $H1$ and $H2$ have an edge joining them but no vertex in common because $v$ is a vertex in $H1$ which is itself a subgraph of $C3$, and $v$ is also the only vertex of the path $P$ to lie in $C3$, where $H2=P-\{v\}$. $H1$ and each of $H3$ and $H4$ have an edge joining them but no vertex in common because each of $H3$ and $H4$ contains one of the endpoints $w$ and $z$ of the path $Q$, but no internal vertices of $Q$, where $H1=Q-\{w,z\}$.
\end{proof}

\begin{figure}[htb]
\centering
\subcaptionbox{if $w$ and $z$ lie in separate components of $C2-\{x y, u\}$ \label{sub1}}[9cm]
{
\begin{tikzpicture}
\Vertex[x=0,y=1,size=.5,label=$x$,fontsize=\large]{x}
\Vertex[x=0,y=-1,size=.5,label=$y$,fontsize=\large]{y}
\Vertex[x=1,y=1,size=.5,label=$w$,fontsize=\large]{w}
\Vertex[x=1,y=-1,size=.5,label=$z$,fontsize=\large]{z}
\Vertex[x=2,y=.5,size=.5]{uw}
\Vertex[x=2,y=-.5,size=.5]{uz}
\Vertex[x=3,y=0,size=.5,label=$u$,fontsize=\large]{u}
\Vertex[x=3,y=1,size=.5]{vw}
\Vertex[x=3,y=-1,size=.5]{vz}
\Vertex[x=4,y=0,size=.5]{uv}
\Vertex[x=5,y=0,size=.5,label=$v$,fontsize=\large]{v}
\Edge[label=$e$,position={right=0mm},fontsize=\large](x)(y)
\Edge[label=$H3$,position={below=0mm},fontsize=\large,style={dashed}](x)(w)
\Edge[label=$H4$,position={above=0mm},fontsize=\large,style={dashed}](y)(z)
\Edge[style={dashed}](w)(uw)
\Edge[style={dashed}](z)(uz)
\Edge[](uw)(u)
\Edge[](uz)(u)
\Edge[](w)(vw)
\Edge[](z)(vz)
\Edge[label=$H1$,position={above=0mm},fontsize=\large,,style={dashed}](vw)(v)
\Edge[style={dashed}](vz)(v)
\Edge[label=$H2$,position={above=0mm},fontsize=\large,style={dashed}](u)(uv)
\Edge[](uv)(v)
\end{tikzpicture}
}
\subcaptionbox{if $w$ and $z$ lie in the same component of $C2-\{x y, u\}$ \label{sub2}}[9cm]
{
\begin{tikzpicture}
\Vertex[x=0,y=1,size=.5,label=$x$,fontsize=\large]{x}
\Vertex[x=0,y=-1,size=.5,label=$y$,fontsize=\large]{y}
\Vertex[x=1,y=1,size=.5,label=$w$,fontsize=\large]{w}
\Vertex[x=2,y=.6,size=.5]{ew}
\Vertex[x=3,y=.2,size=.5]{ez}
\Vertex[x=4,y=-.2,size=.5,label=$z$,fontsize=\large]{z}
\Vertex[x=5,y=-.6,size=.5]{uz}
\Vertex[x=6,y=-1,size=.5,label=$u$,fontsize=\large]{u}
\Vertex[x=4,y=-1,size=.5]{yu}
\Vertex[x=3,y=1,size=.5]{vw}
\Vertex[x=5,y=.4,size=.5]{vz}
\Vertex[x=6,y=0,size=.5]{uv}
\Vertex[x=6,y=1,size=.5,label=$v$,fontsize=\large]{v}
\Edge[label=$H3$,position={right=0mm},fontsize=\large,style={dashed}](x)(y)
\Edge[style={dashed}](x)(w)
\Edge[style={dashed}](w)(ew)
\Edge[label=$e$,position={below=0mm},fontsize=\large](ew)(ez)
\Edge[label=$H4$,position={below=0mm},fontsize=\large,style={dashed}](ez)(z)
\Edge[style={dashed}](z)(uz)
\Edge[](uz)(u)
\Edge[style={dashed}](y)(yu)
\Edge[](yu)(u)
\Edge[](w)(vw)
\Edge[](z)(vz)
\Edge[label=$H1$,position={below=0mm},fontsize=\large,style={dashed}](vw)(v)
\Edge[style={dashed}](vz)(v)
\Edge[label=$H2$,position={left=0mm},fontsize=\large,style={dashed}](u)(uv)
\Edge[](uv)(v)
\end{tikzpicture}
}
\caption{Connected subgraphs $H1,H2,H3,H4$ of $H$ represented by dashed lines (which could represent paths of length $0$ or greater) plus particular vertices, separated pairwise by solid edges.}\label{subgraphs}
\end{figure}

\end{document}